\documentclass[11pt]{amsart}
\usepackage{palatino}
\usepackage{amsfonts}
\usepackage{amssymb}
\usepackage{amscd}
\usepackage[breaklinks,bookmarksopen,bookmarksnumbered]{hyperref}

\newtheorem{theorem}{Theorem}[section]

\theoremstyle{definition}
\newtheorem{definition}[theorem]{Definition}
\newtheorem{remark}[theorem]{Remark}

\newtheorem{conjecture/question}[theorem]{Conjecture/Question}

\newtheorem{remark/definition}[theorem]{Remark/Definition}
\newtheorem{terminology/notation}[theorem]{Terminology/Notation}

\def\PP{{\textbf P}}
\def\OO{\mathcal{O}}

\def\cS{\mathcal{S}}

\def\cM{\mathcal{M}}

\def\rr{\overline{\mathcal{R}}}
\def\cZ{\mathcal{Z}}

\def\Pic0{{\rm Pic}^0(X)}

\def\mm{\overline{\mathcal{M}}}
\def\ss{\overline{\mathcal{S}}}

\begin{document}
\title{The birational type of the moduli space of even spin curves}

\author[G. Farkas]{Gavril Farkas}

\address{Humboldt-Universit\"at zu Berlin, Institut F\"ur Mathematik,
10099 Berlin} \email{{\tt farkas@math.hu-berlin.de}}
\thanks{Research  partially supported by an Alfred P. Sloan Fellowship}
\maketitle

The moduli space $\cS_g$ of smooth spin curves parameterizes pairs
$[C, \eta]$, where $[C]\in \cM_g$ is a curve of genus $g$ and
$\eta\in \mbox{Pic}^{g-1}(C)$ is a theta-characteristic. The finite
forgetful map $\pi: \cS_g \rightarrow \cM_g$ has degree $2^{2g}$ and
$\cS_g$ is a disjoint union of two connected components $\cS_g^{+}$
and $\cS_g^{-}$ of relative degrees $2^{g-1}(2^g+1)$ and
$2^{g-1}(2^g-1)$ corresponding to even and odd theta-characteristics
respectively. A compactification $\ss_g$ of $\cS_g$ over $\mm_g$ is
obtained by considering the coarse moduli space of the stack of
stable spin curves of genus $g$ (cf. \cite {C}, \cite{CCC} and
\cite{AJ}). The projection $\cS_g\rightarrow \cM_g$ extends to a
finite branched covering $\pi:\ss_g\rightarrow \mm_g$. In this paper
we determine the Kodaira dimension of $\ss_g^+$:

\begin{theorem}\label{kodaira}
The moduli space $\ss_g^{+}$ of even spin curves is a variety of general type
for $g>8$ and it is uniruled for $g<8$. The Kodaira dimension of $\ss_8^{+}$ is non-negative  \footnote{Building on the results of this paper, we have proved quite recently in joint work with A. Verra, that
$\kappa(\ss_8^+)=0$. Details will appear later.}.
\end{theorem}

 It was
classically known that $\ss_2^+$ is rational. The Scorza map
establishes a birational isomorphism between $\ss_3^+$ and $\mm_3$,
 cf. \cite{DK}, hence $\ss_3^+$ is rational. Very recently, Takagi and Zucconi \cite{TZ} showed that $\ss_4^+$ is rational as well.  Theorem \ref{kodaira} can be compared to \cite{FL}
Theorem 0.3: The moduli space $\rr_g$ of Prym varieties of dimension
$g-1$ (that is, non-trivial square roots of $\OO_C$ for each $[C]\in \cM_g$) is
of general type when $g>13$ and $g\neq 15$. On the other hand
$\rr_g$ is unirational for $g< 8$. Surprisingly, the problem of
determining the Kodaira dimension has a much shorter solution for
$\ss_g^{+}$ than for $\rr_g$ and our results are complete.

We describe the strategy to prove that $\ss_g^{+}$ is of general
type for a given $g$. We denote by $\lambda=\pi^*(\lambda)\in
\mbox{Pic}(\ss_g^{+})$ the pull-back of the Hodge class and by
$\alpha_0, \beta_0 \in \mbox{Pic}(\ss_g^{+})$ and $\alpha_{i},
\beta_{i} \in \mbox{Pic}(\ss_g^{+})$ for $1\leq i\leq [g/2]$
boundary divisor classes such that
$$\pi^*(\delta_0)=\alpha_0+2 \beta_0
\\ \mbox{ and } \
\pi^*(\delta_i)=\alpha_i+\beta_i\ \mbox{ for } 1\leq i\leq [g/2]$$
(see Section 2 for precise definitions). Using Riemann-Hurwitz  and
\cite{HM} we find that
$$K_{\ss_g^{+}}\equiv\pi^*(K_{\mm_g})+\beta_0 \equiv
13\lambda-2\alpha_0-3\beta_0-2\sum_{i=1}^{[g/2]}
(\alpha_i+\beta_i)-(\alpha_1+\beta_1).$$ We prove that $K_{\ss_g^+}$
is a big $\mathbb Q$-divisor class by comparing it against the class
of the closure in $\ss_g^+$ of the divisor $\Theta_{\mathrm{null}}$
on $\cS_g^{+}$ of non-vanishing even theta characteristics:
\begin{theorem}\label{thetanull}
The closure in $\ss_g^{+}$ of the divisor
$\Theta_{\mathrm{null}}:=\{[C, \eta]\in \cS_g^{+}: H^0(C, \eta)\neq
0\}$ of non-vanishing even theta characteristics has class equal to
$$\overline{\Theta}_{\mathrm{null}}\equiv
\frac{1}{4}\lambda-\frac{1}{16}\alpha_0-\frac{1}{2}\sum_{i=1}^{[g/2]}
 \beta_i\in \   \mathrm{Pic}(\ss_g^{+}).$$
\end{theorem}
Note that the coefficients of $\beta_0$ and $\alpha_i$ for $1\leq
i\leq [g/2]$ in the expansion of $[\overline{\Theta}_{\mathrm{null}}]$ are
equal to $0$. To prove Theorem
\ref{thetanull}, one can use test curves on $\ss_g^+$  or alternatively, realize $\overline{\Theta}_{\mathrm{null}}$ as the push-forward of
the degeneracy locus of a map of vector bundles of the same rank
defined over a certain Hurwitz scheme covering $\ss_g^+$ and  use \cite{F1} and \cite{F2} to compute the class of this locus. Then we use \cite{FP} Theorem 1.1, to construct for each genus
$3\leq g\leq 22$ an effective divisor class $D\equiv
a\lambda-\sum_{i=0}^{[g/2]} b_i\delta_i\in \mbox{Eff}(\mm_g)$ with
coefficients satisfying the inequalities
\[
\frac{a}{b_0}\leq \begin{cases}
6+\frac{12}{g+1}, & \text{if } g+1 \text{ is composite}\\
7, & \text{if }g=10\\
\frac{6k^2+k-6}{k(k-1)}, & \text{if } g=2k-2\geq 4
\end{cases}
\]
and $b_i/b_0 \geq 4/3$ for $1\leq i\leq [g/2]$. When $g+1$ is
composite we choose for $D$ the closure of the Brill-Noether divisor
of curves with a $\mathfrak g^r_d$, that is, $\cM_{g, d}^r:=\{[C]\in
\cM_g: G^r_d(C)\neq \emptyset\}$ in case when the Brill-Noether
number $\rho(g, r, d)=-1$, and then cf. \cite{EH2}
$$\mm_{g, d}^r\equiv c_{g, d,
r}\Bigl((g+3)\lambda-\frac{g+1}{6}\delta_0-\sum_{i=1}^{[g/2]} i(g-i)
\delta_i\Bigr)\in \mathrm{Pic}(\mm_g).$$ For $g=10$ we take the
closure of the divisor $\mathcal{K}_{10}:=\{[C]\in \cM_{10}: C
\mbox{ lies on a } K3 \mbox{ surface}\}$ (cf. \cite{FP} Theorem
1.6). In the remaining cases, when necessarily $g=2k-2$, we choose
for $D$ the Gieseker-Petri divisor $\overline{\mathcal{GP}}_{g,
k}^1$ consisting of those curves $[C]\in \cM_g$ such that  there
exists a pencil $A\in W^1_k(C)$ such that the multiplication map $$\mu_0(A): H^0(C, A)\otimes H^0(C,
K_C\otimes A^{\vee})\rightarrow H^0(C, K_C)$$ is not an isomorphism, see
\cite{EH2}, \cite{F2}. Having chosen $D$, we form the $\mathbb
Q$-linear combination of divisor classes
$$8\cdot \overline{\Theta}_{\mathrm{null}}+\frac{3}{2b_0}\cdot
\pi^*(D)=\bigl(2+\frac{3a}{2b_0}\bigr)\lambda-2\alpha_0-3\beta_0-\sum_{i=1}^{[g/2]}
\frac{3b_i}{2b_0} \alpha_i-\sum_{i=1}^{[g/2]}
\bigl(4+\frac{3b_i}{2}\bigr)\beta_i\in \mathrm{Pic}(\ss_g^+),$$ from
which we can write $$K_{\ss_g^+}=\nu_g\cdot
\lambda+8\overline{\Theta}_{\mathrm{null}}+\frac{3}{2b_0}\pi^*(D)+\sum_{i=1}^{[g/2]}
\bigl(c_i\cdot \alpha_i+ c_i'\cdot \beta_i),$$ where $c_i, c_i'\geq 0$.
Moreover $\nu_g>0$ precisely when $g\geq 9$, while $\nu_8=0$. Since
the class $\lambda\in \mbox{Pic}(\ss_g^+)$ is big and nef, we obtain
that  $K_{\ss_g^+}$ is a big $\mathbb Q$-divisor class on the normal
variety $\ss_g^+$ as soon as $g>8$. It is proved in \cite{lu2007}
that for $g\geq 4$ pluricanonical forms defined on $\ss_{g,
\mathrm{reg}}^+$ extend to any resolution of singularities
$\widehat{\cS_g^+}\rightarrow \ss_g^+$, which shows that $\ss_g^+$
is of general type whenever $\nu_g>0$ and completes the proof of
Theorem \ref{kodaira} for $g\geq 8$. When $g\leq 7$ we show that
$K_{\ss_g^+}\notin \overline{\mbox{Eff}}(\ss_g^+)$ by constructing a covering
curve $R\subset \ss_g^+$ such that $R\cdot K_{\ss_g^+}<0$, cf.
Theorem \ref{negative}. We then use \cite{BDPP} to conclude that $\ss_g^+$ is uniruled.

I would like to thank the referee for pertinent comments which led to a clearly improved
version of this paper.

\section{The stack of spin curves}
We review a few facts about Cornalba's compactification
$\pi:\ss_g\rightarrow \mm_g$, see \cite{C}. If $X$ is a nodal curve,
a smooth rational component $E\subset X$ is said to be
\emph{exceptional} if $\#(E\cap \overline{X-E})=2$. The curve $X$ is
said to be \emph{quasi-stable} if $\#(E\cap \overline{X-E})\geq 2$ for any smooth rational component $E\subset X$, and moreover any two exceptional components of
$X$ are disjoint. A quasi-stable curve is obtained from a stable
curve by blowing-up each node at most once. We denote by $[st(X)]\in
\mm_g$ the stable model of $X$.

\begin{definition}\label{spinstructures} A \emph{spin curve} of genus
$g$ consists of a triple $(X, \eta, \beta)$, where $X$ is a genus
$g$ quasi-stable curve, $\eta\in \mathrm{Pic}^{g-1}(X)$ is a line
bundle of degree $g-1$ such that $\eta_{E}=\OO_E(1)$ for every
exceptional component $E\subset X$, and $\beta:\eta^{\otimes
2}\rightarrow \omega_X$ is a sheaf homomorphism which is generically
non-zero along each non-exceptional component of $X$.
\newline
A \emph{family of spin curves} over a base scheme $S$ consists of a
triple $(\mathcal{X}\stackrel{f}\rightarrow S, \eta, \beta)$, where
$f:\mathcal{X}\rightarrow S$ is a flat family of quasi-stable
curves, $\eta\in \mathrm{Pic}(\mathcal{X})$ is a line bundle and
$\beta:\eta^{\otimes 2}\rightarrow \omega_{\mathcal{X}}$ is a sheaf
homomorphism, such that for every point $s\in S$ the restriction
$(X_s, \eta_{X_s}, \beta_{X_s}:\eta_{X_s}^{\otimes 2}\rightarrow
\omega_{X_s})$ is a spin curve.
\end{definition}

To describe locally the map $\pi:\ss_g\rightarrow \mm_g$ we follow
\cite{C} Section 5. We fix $[X, \eta, \beta]\in \ss_g$ and set
$C:=st(X)$. We denote by $E_1, \ldots, E_r$ the exceptional
components of $X$ and by $p_1, \ldots, p_r\in C_{\mathrm{sing}}$ the
nodes which are images of exceptional components. The automorphism
group of $(X, \eta, \beta)$ fits in the exact sequence of groups
$$1\longrightarrow \mbox{Aut}_0(X, \eta, \beta)\longrightarrow
\mbox{Aut}(X, \eta, \beta)\stackrel{\mathrm{res}_C}\longrightarrow
\mbox{Aut}(C).$$ We denote by $\mathbb C_{\tau}^{3g-3}$ the versal
deformation space of $(X, \eta, \beta)$ where for $1\leq i\leq r$
the locus $(\tau_i=0)\subset \mathbb C_{\tau}^{3g-3}$ corresponds to
spin curves in which the component $E_i\subset X$ persists.
Similarly, we denote by $\mathbb C_t^{3g-3}=\mbox{Ext}^1(\Omega_C,
\OO_C)$ the versal deformation space of $C$ and denote by
$(t_i=0)\subset \mathbb C_t^{3g-3}$ the locus where the node $p_i\in
C$ is not smoothed. Then around the point $[X, \eta, \beta]$, the
morphism $\pi:\ss_g\rightarrow \mm_g$ is locally given by the map
\begin{equation}\label{local}\frac{\mathbb C_{\tau}^{3g-3}}{\mbox{Aut}(X, \eta,
\beta)}\rightarrow \frac{\mathbb C_t^{3g-3}}{\mbox{Aut}(C)}, \
\mbox{  } \mbox{  } t_i=\tau_i^2 \mbox{ } \ (1\leq i\leq r)\ \mbox{
and  }\mbox{ } t_i=\tau_i \mbox{  } \ (r+1\leq i\leq 3g-3).
\end{equation} From now on we specialize to the case of even spin
curves and describe the boundary of $\ss_g^+$. In the process we
determine the ramification of the finite covering $\pi:
\ss_g^+\rightarrow \mm_g$.

\subsection{The boundary divisors of $\ss_g^+$}\hfill

\vskip 3pt
 If $[X, \eta, \beta]\in \pi^{-1}([C\cup_y D])$
where $[C, y]\in \cM_{i, 1}$ and $[D, y]\in \cM_{g-i, 1}$, then necessarily
$X:=C\cup_{y_1} E\cup_{y_2} D$, where $E$ is an exceptional
component such that $C\cap E=\{y_1\}$ and $D\cap E=\{y_2\}$.
Moreover $$\eta=\bigl(\eta_C, \eta_D, \eta_E=\OO_E(1)\bigr)\in
\mbox{Pic}^{g-1}(X),$$ where $\eta_C^{\otimes 2}=K_C, \eta_D^{\otimes
2}=K_D$. The condition $h^0(X, \eta)\equiv 0 \mbox{ mod } 2$,
implies that the theta-characteristics $\eta_C$ and $\eta_D$ have
the same parity. We denote by $A_i\subset \ss_g^+$ the closure of
the locus corresponding to pairs $([C, y, \eta_C], [D, y, \eta_D])\in
\cS_{i, 1}^+\times \cS_{g-i, 1}^+$ and by $B_i\subset \ss_g^+$ the
closure of the locus corresponding to pairs $([C, y, \eta_C], [D,
y, \eta_D])\in \cS_{i, 1}^-\times \cS_{g-i, 1}^{-}$.
\newline

For a general point $[X, \eta, \beta]\in A_i\cup B_i$  we have that
$\mbox{Aut}_0(X, \eta, \beta)=\mbox{Aut}(X, \eta, \beta)=\mathbb
Z_2$. Using (\ref{local}), the map $\mathbb
C_{\tau}^{3g-3}\rightarrow \mathbb C_t^{3g-3}$ is given by
$t_1=\tau_1^2$ and $t_i=\tau_i$ for $i\geq 2$. Furthermore,
$\mbox{Aut}_0(X, \eta, \beta)$ acts on $\mathbb C_{\tau}^{3g-3}$ via
$(\tau_1, \tau_2, \ldots, \tau_{3g-3})\mapsto (-\tau_1, \tau_2,
\ldots, \tau_{3g-3})$. It follows that $\Delta_i\subset \mm_g$ is
not a branch divisor for $\pi:\ss_g^+\rightarrow \mm_g$ and if
$\alpha_i=[A_i]\in \mathrm{Pic}(\ss_g^+)$ and $\beta_i=[B_i]\in
\mathrm{Pic}(\ss_g^+)$, then for $1\leq i\leq [g/2]$ we have the
relation \begin{equation}\pi^*(\delta_i)=\alpha_i+\beta_i.
\end{equation} Moreover,
$\pi_*(\alpha_i)=2^{g-2}(2^i+1)(2^{g-i}+1)\delta_i$ and
$\pi_*(\beta_i)=2^{g-2}(2^i-1)(2^{g-i}-1)\delta_i$.

For a point $[X, \eta, \beta]$ such that $st(X)=C_{yq}:=C/y\sim q$,
with $[C, y, q]\in \cM_{g-1, 2}$, there are two possibilities
depending on whether $X$ possesses an exceptional component or not.
If $X=C_{yq}$ and $\eta_C:=\nu^*(\eta)$ where $\nu:C\rightarrow X$
denotes the normalization map, then $\eta_C^{\otimes 2}=K_C(y+q)$.
For each choice of $\eta_C\in \mathrm{Pic}^{g-1}(C)$ as above, there
is precisely one choice of gluing the fibres $\eta_C(y)$ and
$\eta_C(q)$ such that $h^0(X, \eta)\equiv 0 \mbox{ mod } 2$. We denote by $A_0$ the
closure in $\ss_g^+$ of the locus of points $[C_{yq}, \eta_C\in
\sqrt{K_C(y+q)}]$ as above and clearly
$\mbox{deg}(A_0/\Delta_0)=2^{2g-2}$.

If $X=C\cup_{\{y, q\}} E$ where $E$ is an exceptional component,
then $\eta_C:=\eta\otimes \OO_C$ is a theta-characteristic on $C$.
Since $H^0(X, \omega)\cong H^0(C, \omega_C)$, it follows that $[C,
\eta_C]\in \cS_{g-1}^{+}$. For $[C, y, q]\in \cM_{g-1, 2}$
sufficiently generic we have that $\mbox{Aut}(X, \eta,
\beta)=\mbox{Aut}(C)=\{\mbox{Id}_C\}$, and then from (\ref{local}) it
follows that $\pi$ is simply branched over such points. We denote by
$B_0\subset \ss_g^+$ the closure of the locus of points
$[C\cup_{\{y, q\}} E, \eta_C\in \sqrt{K_C}, \eta_E=\OO_E(1)]$. If
$\alpha_0=[A_0]\in \mbox{Pic}(\ss_g^+)$ and $\beta_0=[B_0]\in
\mbox{Pic}(\ss_g^+)$,  we then have the relation
\begin{equation}\label{del0}
\pi^*(\delta_0)=\alpha_0+2\beta_0.
\end{equation} Note that $\pi_*(\alpha_0)=2^{2g-2}\delta_0$  and
$\pi_*(\beta_0)=2^{g-2}(2^{g-1}+1)\delta_0$.

\subsection{The uniruledness of $\ss_g^+$ for small $g$.}\hfill

\vskip2pt

We employ a simple negativity argument to determine $\kappa(\ss_g^+)$
for small genus. Using an analogous idea we showed that similarly, for
the moduli space of Prym curves, one has that $\kappa(\rr_g)=-\infty$ for
$g<8$, cf. \cite{FL} Theorem 0.7.
\begin{theorem}\label{negative}
 For $g<8$, the space $\ss_g^+$ is uniruled.
\end{theorem}
\begin{proof}  We start with a fixed $K3$ surface $S$ carrying a Lefschetz pencil of curves
of genus $g$. This induces a fibration
$f:\mathrm{Bl}_{g^2}(S)\rightarrow \PP^1$ and then we set
$B:=\bigl(m_{f}\bigr)_
*(\PP^1)\subset  \mm_g$, where $m_f:\PP^1\rightarrow \mm_g$ is the moduli map $m_f(t):=[f^{-1}(t)]$.
We have the following well-known formulas
on $\mm_g$ (cf. \cite{FP} Lemma 2.4):
$$B\cdot \lambda=g+1,\ B\cdot \delta_0=6g+18,\ \  \mbox{ and } B\cdot \delta_i=0\ \mbox{ for } i\geq 1.$$
We lift $B$ to a pencil $R\subset \ss_g^+$ of spin curves by taking
$$R:=B\times _{\mm_g} \ss_g^+=\{[C_t, \ \eta_{C_{t}}]\in \ss_g^{+}:
[C_{t}]\in B, \eta_{C_{t}}\in \overline{\mbox{Pic}}^{g-1}(C_{t}), t
\in \PP^1 \}\subset \ss_g^+.$$ Using (\ref{del0}) one computes the
intersection numbers with the generators of $\mbox{Pic}(\ss_g^+)$:
$$R\cdot \lambda =(g+1)2^{g-1}(2^g+1), \ R\cdot
\alpha_0=(6g+18)2^{2g-2}\ \mbox{ and } R\cdot
\beta_0=(6g+18)2^{g-2}(2^{g-1}+1).$$ Furthermore, $R$ is disjoint
from all the remaining boundary classes of $\ss_g^+$, that is,
$R\cdot \alpha_i=R\cdot \beta_i=0$ for $1\leq i\leq [g/2]$. One
verifies that $R\cdot K_{\ss_g^+}<0$ precisely when $g\leq 7$. Since
$R$ is a covering curve for $\ss_g^{+}$ in the range $g\leq 7$, we
find  that $K_{\ss_g^+}$ is not pseudo-effective, that is, $K_{\ss_g^+}\in \overline{\mbox{Eff}}(\ss_g^+)^c$. Pseudo-effectiveness of the canonical bundle is a birational property for normal varieties, therefore the canonical bundle of any smooth model of $\ss_{g}^+$ lies outside the pseudo-effective cone as well. One can apply \cite{BDPP} Corollary 0.3, to conclude that $\ss_{g}^+$ is uniruled for $g\leq 7$.
\end{proof}

\section{The geometry of the divisor $\overline{\Theta}_{\mathrm{null}}$}

We compute the class of the divisor
$\overline{\Theta}_{\mathrm{null}}$ using test curves. The same calculation can be carried out using techniques developed in \cite{F1}, \cite{F2} to
calculate push-forwards of tautological classes from stacks of limit
linear series $\mathfrak g^r_d$ (see also Remark \ref{hurw}).

For $g\geq 9$, Harer \cite{H} has showed that  $H^2(\mathcal{S}_g^+, \mathbb Q)\cong \mathbb Q$.
The range for which this result holds has been recently improved to $g\geq 5$ in \cite{P}. In particular, it follows that $\mbox{Pic}(\ss_g^+)_{\mathbb Q}$ is generated by the classes $\lambda$, $\alpha_i, \beta_i$ for $i=0, \ldots, [g/2]$.  Thus we can expand the divisor class
$\overline{\Theta}_{\mathrm{null}}$ in terms of the generators of the Picard group
\begin{equation}\label{exp}
\overline{\Theta}_{\mathrm{null}}\equiv \bar{\lambda}\cdot
\lambda-\bar{\alpha}_0\cdot \alpha_0-\bar{\beta}_0\cdot
\beta_0-\sum_{i=1}^{[g/2]}\bigl( \bar{\alpha}_i \cdot
\alpha_i+\bar{\beta}_i \cdot \beta_i\bigr)\in
\mathrm{Pic}(\ss_g^+)_{\mathbb Q},
\end{equation} and determine the coefficients
$\bar{\lambda}, \bar{\alpha}_0, \bar{\beta}_0, \bar{\alpha}_i$ and
$\bar{\beta}_i\in \mathbb Q$ for $1\leq i\leq [g/2]$.

\begin{remark}\label{hurw} To show that the class $[\Theta_{\mathrm{null}}]\in \mbox{Pic}(\mathcal{S}_g^+)_{\mathbb Q}$ is a multiple of $\lambda$ and thus, the expansion (\ref{exp}) makes sense for all $g\geq 3$, one does not need to know that $\mbox{Pic}(\cS_g^+)_{\mathbb Q}$ is infinite cyclic. For instance, for even $g=2k-2\geq 4$, we note that, via the base point free pencil trick,  $[C, \eta]\in \Theta_{\mathrm{null}}$ if and only if the multiplication map $$\mu_C(A, \eta): H^0(C, A)\otimes H^0(C, A\otimes \eta)\rightarrow H^0(C, A^{\otimes 2}\otimes \eta)$$ is not an isomorphism for a base point free pencil $A\in W^1_k(C)$. We set $\widetilde{\cM}_g$ to be the open subvariety consisting of curves $[C]\in \cM_g$ such that $W_{k-1}^1(C)=\emptyset$ and denote by $\sigma: \mathfrak G^1_k\rightarrow \widetilde{\cM}_g$  the Hurwitz scheme of pencils $\mathfrak g^1_k$ and by $$\tau:\mathfrak G^1_k\times_{\widetilde{\cM}_g}\cS_g^+\rightarrow \cS_g^+, \ \ u:\mathfrak G^1_k\times_{\widetilde{\cM}_g} \cS_g^+\rightarrow \mathfrak G^1_k$$  the (generically finite) projections. Then $\Theta_{\mathrm{null}}=\tau_*(\mathcal{Z})$, where $$\cZ=\{[A, C, \eta]\in \mathfrak G^1_k\times_{\widetilde{\cM}_g} \cS_g^+: \mu_C(A, \eta) \mbox{ is not injective}\}.$$ Via this determinantal presentation, the class of the divisor $\cZ$ is expressible as a combination of $\tau^*(\lambda), u^*(\mathfrak a), u^*(\mathfrak b)$, where $\mathfrak{a}, \mathfrak{b}\in \mathrm{Pic}(\mathfrak G^1_k)_{\mathbb Q}$ are the tautological classes defined in e.g. \cite{FL} p.15. Since $\tau_*(u^*(\mathfrak a))=\pi^*(\sigma_*(\mathfrak a))$ (and similarly for the class $\mathfrak b$), the conclusion follows. For odd genus $g=2k-1$, one uses a similar argument replacing $\mathfrak G^1_k$ with any generically finite covering of $\cM_g$ given by a Hurwitz scheme (for instance, we take the space of pencils $\mathfrak g^1_{k+1}$ with a triple ramification point).

\end{remark}

We start the proof of Theorem \ref{thetanull} by determining the
coefficients of $\alpha_i$ and $\beta_i$ ($i\geq 1)$ in the
expansion of $[\overline{\Theta}_{\mathrm{null}}]$.

\begin{theorem}\label{di} We fix integers $g\geq 3$ and $1\leq i\leq [g/2]$.
The coefficient of $\alpha_i$ in the expansion of
$[\overline{\Theta}_{\mathrm{null}}]$ equals $0$, while the
coefficient of $\beta_i$ equals $-1/2$. That is, $\bar{\alpha}_i=0$ and
$\bar{\beta}_i=1/2$.
\end{theorem}
\begin{proof} For each integer $2\leq i\leq g-1$, we fix general curves $[C]\in \cM_{i}$ and $[D, q]\in
\cM_{g-i, 1}$ and consider the test curve $C^i:=\{C\cup _{y\sim q}
D\}_{y\in C}\subset \Delta_i\subset \mm_g$. We lift $C^i$ to test
curves $F_i\subset A_i$ and $G_i\subset B_i$ inside $\ss_g^+$
constructed as follows. We fix even (resp. odd)
theta-characteristics $\eta_C^+\in \mbox{Pic}^{i-1}(C)$ and
$\eta_D^+\in \mbox{Pic}^{g-i-1}(D)$ (resp. $\eta_C^-\in
\mathrm{Pic}^{i-1}(C)$ and $\eta_D^-\in \mathrm{Pic}^{g-i-1}(D)$).

If $E\cong \PP^1$ is an exceptional component, we define the family
$F_i$ (resp. $G_i$) as consisting of spin curves
$$F_i:=\bigl\{t:=[C\cup_y E\cup _q D,\  \eta_C=\eta_C^+, \eta_E=\OO_E(1), \eta_D=\eta_D^+]\in \ss_g^{+}: y\in
C\bigr\}$$ and $$G_i:=\bigl\{t:=[C\cup_y E\cup _q D,\  \eta_C=\eta_C^-,
\eta_E=\OO_E(1), \eta_D=\eta_D^-]\in \ss_g^{+}: y\in C\bigr\}.$$ Since
$\pi_*(F_i)=\pi_*(G_i)=C^i$, clearly $F_i\cdot \alpha_i=C^i\cdot
\delta_i=2-2i, F_i\cdot \beta_i=0$ and $F_i$ has intersection number
$0$ with all other generators of $\mbox{Pic}(\ss_g^+)$. Similarly
$$G_i\cdot \beta_i=2-2i,\ G_i\cdot \alpha_i=0,  \ G_i\cdot \lambda=0,$$ and $G_i$ does not
intersect the remaining boundary classes in $\ss_g^+$.

Next we determine $F_i\cap \overline{\Theta}_{\mathrm{null}}$. Assume that a
point $t\in F_i$ lies in
$\overline{\Theta}_{\mathrm{null}}$. Then there exists a family of even spin curves $(f:\mathcal{X}\rightarrow S, \eta, \beta)$, where
$S=\mbox{Spec}(R)$, with $R$ being a discrete valuation ring and $\mathcal{X}$ is a smooth surface, such that,   if $0, \xi\in S$ denote the special and the generic point of $S$ respectively and $X_{\xi}$ is the generic fibre of $f$, then $$h^0(X_{\xi}, \eta_{\xi})\geq 2, \ h^0(X_{\xi}, \eta_{\xi})\equiv 0 \mbox{ mod 2}, \  \ \eta_{\xi}^{\otimes 2}\cong \omega_{X_{\xi}}\ \mbox{ and } \bigl(f^{-1}(0), \eta_{f^{-1}(0)}\bigr)=t\in \ss_g^+.$$ Following the procedure described in \cite{EH1} p. 347-351,  this data produces  a limit
linear series $\mathfrak g^1_{g-1}$ on $C\cup D$, say $$l:=\Bigl(l_C=(L_C, V_C), l_D=(L_D, V_D)\Bigr)\in
G^1_{g-1}(C)\times G^1_{g-1}(D),$$ such that the underlying  line bundles  $L_C$ and $L_D$ respectively,  are obtained from the line bundle $(\eta_C^+, \eta_E, \eta_D^+)$ by dropping the $E$-aspect and then tensoring the line bundles $\eta_C^+$ and $\eta_D^+$ by line bundles supported at the points $y\in C$ and $q\in D$ respectively. For degree reasons, it follows that $L_C=\eta_C^+\otimes \OO_C((g-i) y)$ and $L_D=\eta_D^+\otimes \OO_D(i q)$.
Since both $C$ and $D$ are general in their respective moduli spaces, we have that $H^0(C, \eta_C^+)=0$ and $H^0(D, \eta_D^+)=0$. In particular
$a_1^{l_C}(y)\leq g-i-1$ and $a_0^{l_D}(q)<a_1^{l_D}(q)\leq i-1$, hence $a_1^{l_C}(y)+a_0^{l_D}(q)\leq g-2$, which contradicts the definition of a limit $\mathfrak g^1_{g-1}$.  Thus $F_i\cap
\overline{\Theta}_{\mathrm{null}}=\emptyset$. This implies that $\bar{\alpha_i}=0$, for all $1\leq i\leq [g/2]$ (for $i=1$, one uses instead the curve $F_{g-1}\subset A_1$ to reach the same conclusion). \vskip 4pt
 Assume that
$t\in G_i\cap \overline{\Theta}_{\mathrm{null}}$. By the same argument as above, retaining also the notation, there is an induced limit linear series on $C\cup D$, $$(l_C, l_D)\in G^1_{g-1}(C)\times G^1_{g-1}(D),$$ where  $L_C=\eta_C^{-}\otimes \OO_C((g-i)y)$ and $L_D=\eta_D^-\otimes \OO_D(iq)$. Since $[C]\in \cM_i$ and $[D, q]\in \cM_{g-i, 1}$ are both general, we may assume that $h^0(D, \eta_D^-)=h^0(C, \eta_C^-)=1$, $q\notin \mbox{supp}(\eta_D^-)$ and that $\mbox{supp}(\eta_C^-)$ consists of $i-1$ distinct points. In particular $a_1^{l_D}(q)\leq i$, hence $a_0^{l_C}(y)\geq g-1-a_1^{l_D}(q)\geq g-i-1$. Since $h^0(C, \eta_C^-)=1$, it follows that one has in fact equality, that is, $a_0^{l_C}(y)=g-i-1$ and then necessarily $a_1^{l_D}(q)=i$.

 Similarly, $a_1^{l_C}(y)\leq g-i+1$ (otherwise $\mbox{div}(\eta_C^-)\geq 2y$, that is, $\mbox{supp}(\eta_C^-)$ would be non-reduced, a contradiction), thus $a_0^{l_D}(q)\geq i-2$, and the last two inequalities must be equalities as well (one uses that  $h^0\bigl(D, L_D\otimes \OO_D(-(i-1)q)\bigr)=h^0(D, \eta_D^-\otimes \OO_D(q))=1$, that is, $a_0^{l_D}(q)<i-1$). Since $a_1^{l_C}(y)=g-i+1$, we find that $y\in \mbox{supp}(\eta_C^-)$.

To sum up, we have showed that $(l_C, l_D)$ is a refined limit $\mathfrak g^1_{g-1}$
and in fact
\begin{equation}\label{limitgi}
l_D=|\eta_D^-\otimes \OO_D(2q)|+(i-2)\cdot q\in G^1_{g-1}(D), \ \ l_C=|\eta_C^-\otimes \OO_C(y)|+(g-i-1)\cdot y\in G^1_{g-1}(C),
\end{equation}
hence  $a^{l_D}(q)=(i-2, i)$ and $a^{l_C}(y)=(g-i-1, g-i+1)$.

To prove that the intersection between $G_i$ and
$\overline{\Theta}_{\mathrm{null}}$ is transversal, we follow closely \cite{EH3} Lemma 3.4 (see especially the \emph{Remark} on p. 45): The restriction $\overline{\Theta}_{\mathrm{null} \ | G_i}$ is isomorphic, as a scheme, to the variety $\tau: \mathfrak T^1_{g-1}(G_i)\rightarrow G_i$ of limit linear series $\mathfrak g^1_{g-1}$ on the curves of compact type $\{C\cup_{y\sim q} D: y\in C\}$, whose $C$ and $D$-aspects  are obtained by twisting suitably at $y\in C$ and $q\in D$ the fixed theta-characteristics $\eta_C^-$ and $\eta_D^-$ respectively. Following the description of the scheme structure of this moduli space given in \cite{EH1} Theorem 3.3 over an arbitrary base, we find that because $G_i$ consists entirely of singular spin curves of compact type, the scheme $\mathfrak T^1_{g-1}(G_i)$ splits as a product of the corresponding moduli spaces of $C$ and $D$-aspects respectively of the limits $\mathfrak g^1_{g-1}$. By direct calculation we have showed
that $\mathfrak{T}^1_{g-1}(G_i) \cong \mathrm{supp}(\eta_C^-)\times \{l_D\}$. Since $\mathrm{supp}(\eta_C^-)$ is a reduced $0$-dimensional scheme, we obtain  that $\overline{\Theta}_{\mathrm{null} \ | G_i}$ is everywhere reduced.
It follows that $G_i\cdot
\overline{\Theta}_{\mathrm{null}}=\#\mbox{supp}(\eta_C^-)=i-1$ and
then $\bar{\beta_i}=(G_i\cdot
\overline{\Theta}_{\mathrm{null}})/(2i-2)$. This argument does not work for $i=1$, when one uses instead the intersection of $\overline{\Theta}_{\mathrm{null}}$ with $G_{g-1}$, and this  finishes the proof.
\end{proof}

Next we construct two pencils in $\ss_g^+$ which are lifts of the
standard degree $12$ pencil of elliptic tails in $\mm_g$. We fix a
general pointed curve $[C, q]\in \cM_{g-1, 1}$ and a pencil $f:
\mathrm{Bl}_9(\PP^2)\rightarrow \PP^1$ of plane cubics together with
a section $\sigma:\PP^1\rightarrow \mathrm{Bl}_9(\PP^2)$ induced by
one of the base points. We then consider the pencil
$R:=\{[C\cup_{q\sim \sigma(\lambda)} f^{-1}(\lambda)]\}_{\lambda\in
\PP^1} \subset \mm_g$.

We fix an odd theta-characteristic $\eta_C^{-}\in
\mbox{Pic}^{g-2}(C)$ such that $q\notin \mbox{supp}(\eta_C^{-})$ and
$E\cong \PP^1$ will again denote an exceptional component. We define
the family
$$F_0:=\{[C\cup_{q} E\cup_{\sigma(\lambda)} f^{-1}(\lambda),\
 \ \eta_C=\eta_C^{-},\  \eta_E=\OO_E(1),\ \
\eta_{f^{-1}(\lambda)}=\OO_{f^{-1}(\lambda)}]: \lambda\in
\PP^1\}\subset \ss_g^+.$$ Since $F_0\cap A_1=\emptyset$, we find
that $F_0\cdot \beta_1=\pi_*(F_0)\cdot \delta_1=-1$. Similarly,
$F_0\cdot \lambda=\pi_*(F_0)\cdot \lambda=1$ and obviously $F_0\cdot
\alpha_i=F_0\cdot \beta_i=0$ for $2\leq i\leq [g/2]$. For each of
the $12$ points $\lambda_{\infty}\in \PP^1$ corresponding to
singular fibres of $R$, the associated $\eta_{\lambda_{\infty}}\in
\overline{\mathrm{Pic}}^{g-1}(C\cup E \cup
f^{-1}(\lambda_{\infty}))$ are actual line bundles on $C\cup E\cup
f^{-1}({\lambda_{\infty}})$ (that is, we do not have to blow-up the
extra node). Thus we obtain that $F_0\cdot \beta_0=0$, therefore
$F_0\cdot \alpha_0=\pi_*(F_0)\cdot \delta_0=12$. \vskip 4pt

We also fix an even theta-characteristic $\eta_C^{+}\in
\mathrm{Pic}^{g-2}(C)$ and consider the degree $3$ branched covering
$\gamma: \ss_{1, 1}^+\rightarrow \mm_{1, 1}$ forgetting the spin
structure. We define the pencil
$$G_0:=\{\bigl[C\cup_q E\cup_{\sigma(\lambda)} f^{-1}(\lambda), \ \eta_C=\eta_C^{+},
\  \eta_E=\OO_E(1), \eta_{f^-1(\lambda)}\in \gamma^{-1}
[f^{-1}(\lambda)]\bigr]:\lambda\in \PP^1\}\subset \ss_g^+.$$ Since
$\pi_*(G_0)=3R$, we have that $G_0\cdot \lambda=3$. Obviously
$G_0\cdot \beta_0=G_0\cdot \beta_1=0$, hence $G_0\cdot
\alpha_1=\pi_*(G_0)\cdot \delta_1=-3$. The map $\gamma:\ss_{1,
1}^+\rightarrow \mm_{1, 1}$ is simply ramified over the point
corresponding to $j$-invariant $\infty$. Hence, $G_0\cdot
\alpha_0=12$ and $G_0\cdot \beta_0=12$, which is consistent with
formula (\ref{del0}). \vskip 4pt The last pencil we construct lies
in the boundary divisor $B_0\subset \ss_g^+$: Setting $E\cong \PP^1$
for an exceptional component, we define
$$H_0:=\{[C\cup_{\{y, q\}} E, \ \eta_C=\eta_C^+, \ \eta_E=\OO_E(1)]:
y\in C\}\subset \ss_g^+.$$ The fibre of $H_0$ over the point $y=q\in
C$ is the even spin curve $$\bigl[C\cup_q E'\cup _{q'}
E''\cup_{\{q'', y''\}} E, \ \eta_C=\eta_C^{+},
\eta_{E'}=\OO_{E'}(1), \eta_E=\OO_E(1),
\eta_{E''}=\OO_{E''}(-1)\bigr],$$ having as stable model
$[C\cup _q E_{\infty}]$, where $E_{\infty}:=E''/y''\sim
q''$ is the rational nodal curve corresponding to $j=\infty$.
Here $E', E''$ are rational curves, $E'\cap E''=\{q'\}$,
 $E\cap E''=\{q'', y''\}$ and the stabilization map for $C\cup E\cup E'\cup E''$ contracts the components $E'$ and $E$, while identifying $q''$ and
 $y''$.

We find that $H_0\cdot \lambda=0, H_0\cdot \alpha_i=H_0\cdot
\beta_i=0$ for $2\leq i\leq [g/2]$. Moreover $H_0\cdot \alpha_0=0$,
hence $H_0\cdot \beta_0=\frac{1}{2} \pi_*(H_0)\cdot \delta_0=1-g$.
Finally, $H_0\cdot \alpha_1=1$ and $H_0\cdot \beta_1=0$.

\vskip 5pt

\begin{theorem}\label{penc} If $F_0, G_0, H_0\subset \ss_g^{+}$ are the families
of spin curves defined above, then $$F_0\cdot
\overline{\Theta}_{\mathrm{null}}=G_0\cdot
\overline{\Theta}_{\mathrm{null}}=H_0 \cdot
\overline{\Theta}_{\mathrm{null}}=0.$$
\end{theorem}
\begin{proof}
From the limit linear series argument in the proof of Theorem
\ref{di} we get that the assumption $F_0\cap
\overline{\Theta}_{\mathrm{null}}\neq \emptyset$ implies that $q\in
\mbox{supp}(\eta_C^{-})$, a contradiction. Similarly, we have that
$G_0\cap \overline{\Theta}_{\mathrm{null}}= \emptyset$ because
$[C]\in \cM_{g-1}$ can be assumed to have no even
theta-characteristics $\eta_C^{+}\in \mbox{Pic}^{g-2}(C)$ with
$h^0(C, \eta_C^{+})\geq 2$, that is $[C, \eta_C^{+}]\notin
\overline{\Theta}_{\mathrm{null}}\subset \ss_{g-1}^+$. Finally, we
assume that there exists a point $[X:=C\cup_{\{y, q\}} E,
\eta_C=\eta_C^+, \eta_E=\OO_E(1)]\in H_0\cap
\overline{\Theta}_{\mathrm{null}}$. Then certainly $h^0(X,
\eta_X)\geq 2$ and from the Mayer-Vietoris sequence on $X$ we find
that $$H^0(X, \eta_X)=\mbox{Ker}\{H^0(C, \eta_C)\oplus H^0(E,
\OO_E(1))\rightarrow \mathbb C^2_{y, q}\},$$ hence
$h^0(C,\eta_C)=h^0(X, \eta_X)\geq 2$. This contradicts the
assumption that $[C]\in \cM_{g-1}$ is general. A similar argument
works for the special point in $H_0\cap \pi^{-1}(\Delta_1)$, hence
$H_0\cdot \overline{\Theta}_{\mathrm{null}}=0$.
\end{proof}

\vskip 4pt
 \noindent {\emph{Proof of Theorem \ref{thetanull}}.}
Looking at the expansion of $[\overline{\Theta}_{\mathrm{null}}]$,
Theorem \ref{penc} gives the relations
$$F_0\cdot \overline{\Theta}_{\mathrm{null}}=\bar{\lambda}-12\bar{\alpha}_0+\bar{\beta}_1=0,\ \
G_0\cdot
\overline{\Theta}_{\mathrm{null}}=3\bar{\lambda}-12\bar{\alpha}_0-12\bar{\beta}_0+3\bar{\alpha}_1=0$$
$$\mbox{ and } \ H_0\cdot \overline{\Theta}_{\mathrm{null}}=(g-1)\bar{\beta}_0-\bar{\alpha}_1=0.$$ Since we have
already computed $\bar{\alpha}_i=0$ and $\bar{\beta}_i=1/2$ for
$1\leq i\leq [g/2]$, (cf. Theorem \ref{di}), we obtain that
$\bar{\lambda}=1/4, \bar{\alpha}_0=1/16$ and $\bar{\beta}_0=0$. This
completes the proof. \hfill $\Box$ \vskip 3pt
 A consequence of
Theorem \ref{thetanull} is a new proof of the main result from
\cite{T}:

\begin{theorem}
If $\cM_g^1$ is the locus of curves $[C]\in \cM_g$ with a vanishing
theta-null then its closure has class equal to
$$\mm_g^1\equiv
2^{g-3}\Bigl((2^g+1)\lambda-2^{g-3}\delta_0-\sum_{i=1}^{[g/2]}
(2^{g-i}-1)(2^i-1)\delta_i\Bigr)\in \mathrm{Pic}(\mm_g).$$
\end{theorem}
\begin{proof} We use the scheme-theoretic equality
$\pi_*(\overline{\Theta}_{\mathrm{null}})=\mm_{g}^1$ as well as the
formulas $\pi_*(\lambda)=2^{g-1}(2^g+1)\lambda,\
\pi_*(\alpha_0)=2^{2g-2}\delta_0, \
\pi_*(\beta_0)=2^{g-2}(2^{g-1}+1)\delta_0,\ \
\pi_*(\alpha_i)=2^{g-2}(2^i+1)(2^{g-i}+1)\delta_i$ and
$\pi_*(\beta_i)=2^{g-2}(2^i-1)(2^{g-i}-1)\delta_i$ valid for $1\leq
i\leq [g/2]$.
\end{proof}

\end{document}